\documentclass[letterpaper, 10 pt, conference]{ieeeconf}

\IEEEoverridecommandlockouts
\usepackage{mathematics}
\usepackage{block-diagram}
\usepackage[hidelinks]{hyperref}
\usepackage[nameinlink]{cleveref}
\usepackage{cite}
\usepackage{lipsum}
\usepackage{parskip}
\usepackage{mathrsfs}
\let\labelindent\relax
\usepackage{enumitem}

\Crefname{figure}{Fig.}{Figs.} 
\crefname{assumption}{Assumption}{Assumptions}
\crefname{problem}{Problem}{Problems}

\newcommand*{\QEDB}{\hfill\ensuremath{\square}}

\definecolor{matlabblue}{rgb}{0.0000, 0.4470, 0.7410}
\definecolor{matlaborange}{rgb}{0.8500, 0.3250, 0.0980}
\definecolor{mygray}{rgb}{0.7, 0.7, 0.7}
\definecolor{mygreen}{HTML}{13ec13}

\title{\LARGE\bf Temporal Variabilities Limit Convergence Rates in \\
Gradient-Based Online Optimization}

\author{Bryan Van Scoy\and Gianluca Bianchin%
\thanks{This material is based upon work supported in part by the National Science Foundation under Award No. 2347121 and in part by the FRFS WEL-T Investigator Programme. Any opinions, findings and conclusions or recommendations expressed in this material are those of the authors and do not necessarily reflect the views of the National Science Foundation.}%
\thanks{%
B.~Van~Scoy is with the Dept. of Electrical and Computer Engineering, Miami University, OH~45056, USA. Email: \texttt{bvanscoy@miamioh.edu}%
}
\thanks{%
G.~Bianchin is with the ICTEAM institute and the Department of Mathematical Engineering (INMA) at the University of Louvain, Belgium.
Email: \texttt{gianluca.bianchin@uclouvain.be}}%
\hspace{-1cm}}

\begin{document}

\maketitle

\begin{abstract}
This paper investigates the fundamental performance limits of gradient-based 
algorithms for time-varying optimization.
Leveraging the internal model principle and root locus techniques, 
we show that temporal variabilities impose intrinsic limits on the 
achievable rate of convergence. For a problem with condition ratio $\kappa$ 
and time variation whose model has degree $n$, we show that the worst-case 
convergence rate of any  minimal-order gradient-based algorithm is
$\rho_\text{TV} = \bigl(\tfrac{\kappa-1}{\kappa+1}\bigr)^{1/n}$.
This bound reveals a fundamental tradeoff between problem conditioning,
temporal complexity, and rate of convergence.
We further construct explicit controllers that attain the bound for low-degree models of time variation.
\end{abstract}

\section{Introduction}

Time-varying optimization problems provide a natural framework to describe 
decision-making tasks in which objectives and/or constraints evolve dynamically 
over time. Such problems arise in diverse engineering domains, including online 
learning and streaming data in machine learning~\cite{YC-YZ:20,AR-KS:13}, adaptive 
filtering in signal processing~\cite{FJ-AR:12}, and trajectory planning or 
model predictive control in robotics~\cite{AD-VC-AG-GR-FB:23}.

Historically, the study of time-varying optimization has relied on 
continuity 
arguments with respect to static formulations: when the temporal variations 
of the problem are sufficiently slow, algorithms developed for 
static problems produce nearly optimal solutions in the 
time-varying setting~\cite{AS-ED-SP-GL-GG:20}. This reasoning suggests 
that algorithms for time-varying problems may exhibit the same fundamental
performance characteristics as those for time-invariant problems.
%
For example, it is well established that the best achievable rate of gradient descent on time-invariant problems is $\rho_\text{TI} = \tfrac{\kappa-1}{\kappa+1}$, where~$\kappa$ denotes the condition ratio of the problem~\cite{LL-BR-AP:16}.
%
Accordingly, one may be tempted to conclude that algorithms for time-varying 
problems can attain the same rate, provided that the temporal variability is `slow 
enough.'
In contrast, we prove a fundamental relationship between the convergence rate $\rho$, the condition ratio $\kappa$, and the number $n$ of modes that characterize the temporal variability of the problem. Since there is no direct relation between $n$ and {`how fast'} 
the problem varies, but rather with the \emph{complexity} of its temporal 
structure, our results reveal fundamental differences between static 
optimization problems and their dynamic counterparts.


{\it Contributions.}
Our main contributions are as follows:
\begin{enumerate}[topsep=-4pt,itemsep=-2pt]
    \item We provide a fundamental bound on the worst-case convergence rate of minimal-degree controllers for unconstrained quadratic time-varying optimization, which~is 
    \begin{equation}\label{eq:TV}
        \rho_\text{TV} \defeq \biggl(\frac{\kappa-1}{\kappa+1}\biggr)^{1/n},
    \end{equation}
    where $\kappa$ is the condition ratio of the objective, and $n$ is the number of modes in the model of the time variation.
    Since quadratics are a special case of convex, Lipschitz-smooth objectives, this lower bound also applies as fundamental limitations to this broader class.
    \item We use root locus techniques to design explicit controllers for particular models of the time variation.
\end{enumerate}

{\it Related works.}
The literature on time-varying optimization methods can be broadly divided 
into two classes. The first class comprises approaches that ignore or do not 
exploit any model of the temporal variability, and instead solve a sequence 
of static problems~\cite{MZ:03,EH-AA-SK:07,EH:16}.
These methods only 
react after changes are observed, and therefore incur a certain regret and achieve at 
best convergence to a neighborhood of the optimizer~\cite{AS-ED-SP-GL-GG:20}. 
The second class, instead, leverages a model of the temporal evolution of 
the problem to track the optimal trajectory exactly. Indeed, such a model is necessary for exact tracking~\cite{GB-BVS:25-cdc,GB-BVS:24-arxiv}. A prominent example is 
the prediction-correction framework~\cite{YZ-MS:98,MF-SP-VP-AR:17}, where each 
step combines a prediction of the optimizer's evolution with a correction based 
on the current problem. Extensions of these ideas have recently been studied 
under contraction analysis~\cite{AD-VC-AG-GR-FB:23}, sampling-based estimation 
of variability~\cite{MM-JB-JS-PT:24}, and constrained formulations~\cite{AS-ED:17}; 
see also the survey~\cite{AH-ZH-SB-GH-FD:24}. 

There has recently been a spark of interest in using control tools to 
design optimization algorithms. Fundamental results have been derived 
in~\cite{GB-BVS:25-cdc} for discrete-time problems and~\cite{GB-BVS:24-arxiv} for 
continuous-time ones. 
For problems with quadratic objectives, root locus and the internal model principle can be used to analyze and design the optimization method~\cite{NB-RC-SZ:24}. Recently, \cite{UC-NB-RC-SZ:25} 
have investigated constrained settings and stochastic problems. 
Particularly relevant to this work are~\cite{AW-IP-VU-IS:25,AXW-IRP-IS:25}, 
which considered quadratic objectives with polynomial temporal variabilities. 
In~\cite{AXW-IRP-IS:25}, the authors use Nevanlinna--Pick interpolation to establish a fundamental lower bound of 
$\bigl(\tfrac{\sqrt{\kappa}-1}{\sqrt{\kappa}+1}\bigr)^{1/n}$ 
(cf.~\eqref{eq:TV}). 
These results, however, are restricted to polynomial temporal variabilities 
and focus on accelerated methods, which are known to potentially fail to 
achieve global convergence on more general objectives 
(beyond quadratics)~\cite{LL-BR-AP:16}, and are also susceptible to noise 
amplification in gradient evaluations~\cite{HM-MR-MRJ:25}. For these reasons, 
we restrict our attention to the class of non-accelerated (minimal-order) methods 
and general temporal variabilities.


{\it Notation:} 
We denote by $\natural$ and $\real$ the sets of natural and real 
numbers, respectively; by $\real[z]$ the space of real-coefficient 
polynomials in $z$; and by $\real^d[z]$ the space of $d$-dimensional 
vector polynomials in $z$ with real coefficients.

\section{Problem Setup}

We consider time-varying unconstrained optimization problems, consisting of 
minimizing the quadratic objective
\begin{equation}\label{eq:objective}
    f_k(x) = \tfrac{1}{2} x^\tp A x + b_k^\tp x
\end{equation}
with time indexed by $k \in \natural.$
Here, $x \in \mathbb{R}^d$ is the decision variable, \(A \in \mathbb{R}^{d \times d}\) is a symmetric time-invariant matrix, and 
$b = \{b_k\}_{k \in \mathbb{N}}$ with $b_k \in \mathbb{R}^d$ is a time-varying 
parameter.

\begin{remark}[Quadratic objective functions]
We focus on the class of quadratic objectives as they provide a structured 
framework to derive bounds on the worst-case convergence rate.
Since quadratics are a special case of convex, Lipschitz-smooth 
objectives, and our goal is to establish lower bounds on the convergence 
rate, the forthcoming estimates also serve as fundamental limitations for 
this broader function class. 
This is in line with the time-invariant case, 
where~\cite[Sec.~2.1.4]{YN:18} shows that the optimal convergence rate 
achievable by any iterative algorithm on problems with convex and 
Lipschitz-smooth loss is attained for quadratics.~
\QEDB\end{remark}

We make the following assumptions throughout.

\begin{assumption}[Eigenvalues of $A$]\label{assumption:eigenvalues}
  The matrix $A$ has eigenvalues in the closed interval $[\mu,L]$ with $0 < \mu < L$. Moreover, the parameters $\mu$ and $L$ are known.
\QEDB\end{assumption}

By \Cref{assumption:eigenvalues}, the cost~\eqref{eq:objective} is 
strongly convex with parameter $\mu$, and the gradient is Lipschitz smooth with parameter $L$.
Considering problems of this class is a standard assumptions in 
optimization~\cite{EH:16}, which has been widely used in related 
works~\cite{GB-JC-JP-ED:21-tcns,AH-SB-GH-FD:20,GC-NM-GN:24}.
In what follows, we let $\kappa\defeq L/m$ denote the {\it condition ratio}
of the objective in~\eqref{eq:objective}.


We make the following assumption on $B(z)$, the $\Z$-transform of the time-varying sequence $b = \{b_k\}_{k\in\natural}.$

\begin{assumption}[Model of time variation]\label{assumption:model}
$B(z)$ is a rational function of $z$ with all poles in $\disk_{=1} \defeq \{z : |z|=1\}$.
\QEDB\end{assumption}

Assumption~\ref{assumption:model} specifies the class of temporal 
variations of the parameter \(b\) under consideration. Intuitively, signals 
whose $\Z$-transform is rational in $z$ are signals generated by 
finite-order linear recurrences (i.e., by LTI systems with a finite number of 
states) with poles that lie on the unit circle. This class includes a wide 
variety of signals including causal periodic sequences 
(sinusoids, square waves, etc.) and polynomial sequences 
(constant, ramp, parabolic, etc.). Similar assumptions have been employed 
in related works~\cite{NB-RC-SZ:24,AXW-IRP-IS:25}.

By \Cref{assumption:model}, $B(z)$ admits the representation
\begin{align}\label{eq:Bz}
    B(z) = \frac{B_N(z)}{m(z)},
\end{align}
where $B_N(z) \in \real^d[z]$ and $m(z) \in \real[z]$. Without loss of 
generality, we use the notation
\begin{align*}
    m(z) = z^n + \sum_{i=0}^{n-1} m_i z^i, \qquad m_i \in \real,
\end{align*}
where $n$ is the number of poles\footnote{Note that~\eqref{eq:Bz} does not restrict the poles of \(B(z)\) to be 
identical for each component, since \(m(z)\) can be 
chosen as the polynomial whose root set is the union of the root sets of the 
individual components of \(B(z)\).} of $B(z)$.

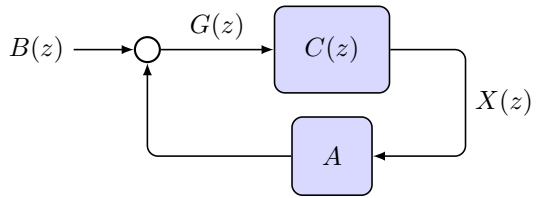
\begin{figure}[t!]
  \centering
  \begin{tikzpicture}
      \node [block] (C) {$C(z)$};
      \node [block,below=3mm of C.south,anchor=north] (A) {$A$};
      \node [sum,left=15mm of C.west,anchor=east] (s) {};
    
      \draw (s) -- node[pos=0.5,anchor=south] {$G(z)$} (C);
      \draw (C.east) -- ++(10mm,0) -- node[pos=0.5,anchor=west] {$X(z)$} (\currentcoordinate |- A) -- (A);
      \draw (A) -| (s);
      \draw [<-] (s.west) -- node[pos=1,anchor=east] {$B(z)$} ++(-8mm,0);
    \end{tikzpicture}
  \caption{Structure of the gradient-basing optimization algorithms, as a block-diagram in the frequency domain. See~\eqref{eq:tf_00}.}
  \label{fig:controller}
\end{figure}

In line with the literature~\cite{YN:18}, we focus on gradient-type algorithms 
for minimizing~\eqref{eq:objective}; that is, algorithms that have 
access to oracle evaluations of the gradient of~\eqref{eq:objective}:
\begin{align}\label{eq:gradient}
(k, x) \ \mapsto\ \nabla f_k(x) = Ax + b_k .
\end{align}
We restrict our attention to optimization algorithms whose iterates are obtained by 
processing~\eqref{eq:gradient} through Linear Time-Invariant (LTI) filters.
Precisely, let $x_k \in \real^d$ denote the estimate for the 
minimizer of~\eqref{eq:objective} generated by the algorithm\footnote{Since our 
focus is on characterizing the asymptotic convergence rate of the 
method, we henceforth assume that the internal state of the optimization filter is initialized to zero, noting that the framework can be 
extended to nonzero initial conditions by accounting for the free response of 
\(C(z)\), an extension we leave for future work due to space limitations.} 
at time  $k \in \natural,$ and by $X(z)$ its $\Z$-transform. 
Then, we consider algorithms that generate $x_k$ as:
\begin{align}\label{eq:tf_00}
X(z) = C(z) G(z) =C(z) (AX(z) + B(z)),
\end{align}
where $C(z) \in \real^{d\times d}[z]$ is the transfer function of the filter, and $G(z) = A X(z) + B(z)$ is the $\Z$-transform of the gradient $\nabla f_k(x_k)$. This algorithm structure is illustrated in \Cref{fig:controller}. We make the following assumption on the filter throughout, which specifies that $C(z)$ is a linear filter, with 
the additional requirement that it be strictly proper to guarantee real-time 
implementability of the algorithm.

\begin{assumption}[Structure of the optimization filter]\label{assumption:controller}
$C(z)$ is a rational, strictly proper function of $z$.  
\QEDB\end{assumption}

\begin{example}[Gradient descent]
\label{ex:grad_descent}
It is immediate to verify (by applying the $\Z$-transform to both sides of the equation) 
that the gradient-descent algorithm:
$$x_{k+1}=x_k-\alpha \nabla f_k \left(x_k\right),$$
is a particular instance of~\eqref{eq:tf_00} with $C(z)=-\frac{\alpha}{z-1} I_d.$~
\QEDB\end{example}

In what follows, we focus on designing optimization filters with optimal rate 
of convergence and that are of minimal order; we make these two notions formal 
next.

\begin{definition}[Asymptotic tracking]
\label{def:asymptotic_tracking}
We say that~\eqref{eq:tf_00} \textit{asymptotically tracks the minimizer 
of~\eqref{eq:objective}} if the sequence~$\{x_k\}$ satisfies
\[
    \lim_{k\to\infty}\,\|x_k - x_k^\star\| = 0,
\]
where $x_k^\star := -A^{-1} b_k$ is the minimizer of~\eqref{eq:objective}. Moreover, the \textit{root-convergence factor} (or, simply, {\it convergence rate}) is
\[
    \rho \defeq \limsup_{k\to\infty}\,\|x_k - x_k^\star\|^{1/k}. \tag*{\QEDB}
\]
\end{definition}

Note that \Cref{def:asymptotic_tracking} formalizes a notion of {\it exact} 
tracking, whereby $x_k$ reaches $x_k^\star$ with {\it zero error}, 
asymptotically.

\begin{remark}
The minimizer $x_k^\star$ of the quadratic objective~\eqref{eq:objective} has \(\Z\)-transform 
\(X^\star(z) = -A^{-1} B(z) = -\tfrac{1}{m(z)} A^{-1} B_N(z)\). 
Consequently, \(B(z)\) and \(X^\star(z)\) share the same poles. It follows that 
Assumption~\ref{assumption:model} can equivalently be formulated in 
terms of a model for \(x_k^\star\) rather than \(b_k\), aligning our formulation with other models in the literature 
(e.g.,~\cite{AXW-IRP-IS:25}).
~\QEDB\end{remark}

\begin{definition}[Optimization filters of minimal order]\label{def:minimal-order}
Let $C(z)$ be an optimization filter that asymptotically tracks the minimizer 
of~\eqref{eq:objective}. We say that $C(z)$ is of {\it minimal order} if the  
degree of its denominator polynomial is minimal among all filters that 
asymptotically track the minimizer.    
\QEDB\end{definition}

Although higher-order filters could be employed, yielding accelerated 
algorithms~\cite{AD-DS-AT:21}, our focus here is on the 
class of {\it non-accelerated} gradient methods~\cite{YN:18}. 
Motivations for studying non-accelerated methods include that accelerated methods are known to possibly fail to achieve global convergence on more general objectives (beyond quadratics)~\cite{LL-BR-AP:16} and also amplify noise in gradient evaluations~\cite{HM-MR-MRJ:25}.

We are now ready to formalize the objective of this work.

\begin{problem}\label{prob:main}
Determine the optimal worst-case convergence rate achievable by any 
minimal-order optimization algorithm of the form~\eqref{eq:tf_00}, where 
\textit{optimal} is with respect to all optimization filters of minimal order 
satisfying \Cref{assumption:controller}, and \textit{worst-case} is with 
respect to all objectives of the form~\eqref{eq:objective} satisfying 
\Cref{assumption:eigenvalues,assumption:model}.
In addition, construct an optimization filter that attains this rate.
\QEDB\end{problem}




\section{Preliminaries}

In this section, we provide an instrumental characterization of the convergence rate that will enable us to address \Cref{prob:main}. 
We begin with the following result.

\begin{lemma}[Fundamental structure of tracking filters]
\label{lem:necessary+sufficient_tracking}
Let \Cref{assumption:eigenvalues,assumption:model,assumption:controller} hold, and consider optimization filters of the form
\begin{equation}\label{eq:C-minimal}
    C(z) = \frac{C_N(z)}{m(z)}, \qquad C_N(z)\in\real^{d\times d}[z].
\end{equation}
\begin{enumerate}[label=(L\arabic*),leftmargin=*, itemsep=\parskip, topsep=0pt, partopsep=0pt, parsep=0pt] 
    \item \label{Lemma1.1}
    Suppose $C(z)$ is an optimization filter that achieves 
    exact asymptotic tracking. Then, $C(z)$ is of minimal order only if it has the form~\eqref{eq:C-minimal}.
    \item \label{Lemma1.2}
    Suppose $C(z)$ has the form~\eqref{eq:C-minimal} and the roots of 
    $\det(m(z) I - A C_N(z))$ are in $\disk_{<1} \defeq \{z : |z|<1\}.$
    Then, $C(z)$ is an optimization filter of minimal order 
    that achieves exact asymptotic tracking.
\end{enumerate}
\end{lemma}

\begin{proof}
Solving~\eqref{eq:tf_00} for the gradient yields
\[
    G(z) = (I - AC(z))^{-1} B(z).
\]
By \Cref{assumption:controller}, we can write 
$C(z) = \tfrac{C_N(z)}{c_D(z)}$ with 
$C_N(z) \in \real^{d\times d}[z]$ and $c_D(z) \in \real[z]$. 
Substituting this form and~\eqref{eq:Bz} into the expression for the gradient yields
\begin{align}\label{eq:auxG}
    G(z) = \frac{c_D(z)}{m(z)} \bigl(c_D(z) I - AC_N(z)\bigr)^{-1}\,B_N(z).
\end{align}
The filter achieves exact asymptotic tracking if and only if all poles of $G(z)$ are strictly inside the unit circle.
Each entry of $(c_D(z) I - AC_N(z))^{-1}$ is a rational function of $z$, and the denominator polynomial is $\det(c_D(z)I-AC_N(z))$. Since all roots of $m(z)$ are 
marginally stable by \Cref{assumption:model}, the poles of $G(z)$ would lie on the unit 
circle unless the poles introduced by $m(z)$ were canceled by either $c_D(z)$, $B_N(z)$, or the adjugate of $c_D(z)I-AC_N(z)$. Since the roots of $m(z)$ are poles of $B(z)$, they are not canceled by each component of $B_N(z)$. Moreover, they cannot be canceled by the adjugate for all $A$ satisfying \Cref{assumption:eigenvalues}.
Thus, for $C(z)$ to achieve exact asymptotic tracking with minimal order, it's denominator must take the form $c_D(z) = m(z)$ as claimed.
Finally, the choice $C(z)$ in~\eqref{eq:C-minimal} includes precisely this necessary factor 
and no additional pole factors, and hence is of minimal order.
%
\end{proof}


\Cref{lem:necessary+sufficient_tracking} provides necessary and sufficient 
conditions for asymptotic tracking. The statement~\ref{Lemma1.1} provides a 
necessary condition for an optimization filter to be of minimal order; the 
property that $C(z)$ is required to incorporate precisely the same poles as 
$B(z)$ can be interpreted as an instance of the {\it internal model principle 
of time-varying optimization}~\cite{GB-BVS:25-cdc,GB-BVS:24-arxiv}, as it 
captures the requirement that the optimization filter must embed an internal 
model of the temporal variability of the problem (encoded by \(m(z)\)). 
Conversely, the statement~\ref{Lemma1.2} provides a sufficient condition for an 
optimization filter to be of minimal order and achieve exact tracking, requiring 
that all roots of $\det(m(z) I - A C_N(z))$ to be in the open unit disk.
This condition will be used later in this work to construct algorithms that 
address Problem~\ref{prob:main}.

Because we search within the class of optimization filters of minimal order,
driven by the conclusion of~\ref{Lemma1.2}, we assume the optimization filter has the form~\eqref{eq:C-minimal}. Moreover, we also assume that the same filter is applied to each component of the objective function. 

\begin{assumption}[Optimization filter is of minimal order]\label{assumption:controller-minimal}
The optimization filter has the form $C(z) = c(z)I_d$, where $c(z) =\frac{d(z)}{m(z)}$ for some  $d(z) \in \real[z].$
\QEDB\end{assumption}


We next derive an equivalent reformulation of the convergence rate 
(cf. \Cref{def:asymptotic_tracking}) for this class of algorithms that will be instrumental to derive our 
main results.

\begin{lemma}[Characterization of the convergence rate]
\label{lem:characterization_rho}
Suppose \Cref{assumption:controller,assumption:eigenvalues,assumption:controller-minimal,assumption:model} hold.
The optimal worst-case convergence rate achievable by any 
minimal-order optimization algorithm of the form~\eqref{eq:tf_00} is given by the optimal value of the following min-max problem:
\begin{align}\label{eq:optimization-characterization_rho}
    \rho \deq \min_{d(z)\in\real[z]} ~~ \max_{\lambda\in[\mu,L]} ~~ &\ |z| \notag\\
    \text{subject to}~~ & m(z) - \lambda d(z) =0.
\end{align}
\end{lemma}

\begin{proof}
Let $A = V \Lambda V^\tp$ be an eigendecomposition of~$A$, with 
$\Lambda = \diag(\lambda_1,\ldots, \lambda_d)$  and $V$ an orthonormal matrix of 
(real) eigenvectors. Let $\tilde x_k := V^\top x_k$ and $\tilde X(z)$ denote the 
corresponding $\Z$-transform. By projecting~\eqref{eq:tf_00} onto the range 
space of $V$, we obtain the update:
\begin{align*}
    \tilde X(z) = \frac{d(z)}{m(z)} (\Lambda \tilde X(z) + V^\top B(z)).
\end{align*}
With this decomposition, the iterates of~\eqref{eq:tf_00} separate into $d$ 
decoupled equations index by $i=1,\ldots,n$, given by 
\[
    \tilde X_i(s) = \frac{d(z)}{m(z)} (\lambda_i \tilde X_i(z) + v_i^\top B(z)).
\]
Rewriting this equation in input-output form gives
\begin{equation}\label{eq:reduced-system}
    \tilde X_i(z) = \frac{d(z)}{m(z) - \lambda_i d(z)} v_i^\top B(z).
\end{equation}
Therefore, the closed-loop poles of~\eqref{eq:tf_00} are the roots of 
$m(z) - \lambda_i d(z)$. Since the convergence rate of an LTI system is the maximum modulus of the poles of its transfer function, this gives the 
formulation~\eqref{eq:optimization-characterization_rho}.
\end{proof}

Lemma~\ref{lem:characterization_rho} provides a mathematical reformulation of 
the optimal worst-case convergence rate achievable by any minimal-order 
optimization algorithm. The outer minimization over the polynomial $d(z)$ captures a search over all 
optimization filters of minimal order that satisfy \Cref{assumption:controller}, 
while the inner maximization over $\lambda \in [\mu, L]$ accounts for 
the worst-case scenario over the entire class of objective functions consistent 
with \Cref{assumption:eigenvalues}.

    

\section{Bound on the worst-case convergence rate}

We now provide a bound on the worst-case convergence rate achievable by any minimal-order algorithm as in \Cref{fig:controller} that asymptotically tracks the minimizer of the quadratic objective~\eqref{eq:objective}. Our bound depends only on the condition ratio $\kappa$ of the objective function and the degree $n$ of the model of the time variation. Our main result is the following.

\begin{theorem}[Bound on worst-case convergence rate]
\label{thm:lower_bound}
Let \Cref{assumption:eigenvalues,assumption:model,assumption:controller,assumption:controller-minimal} hold. The worst-case convergence rate achievable by any minimal-order optimization algorithm is bounded below by $\rho_\text{TV}$ defined in~\eqref{eq:TV}, where $\kappa\defeq L/\mu$ is the condition ratio and $n$ is the degree of the model $m(z)$.
\end{theorem}

\begin{proof}
Suppose \Cref{assumption:eigenvalues,assumption:model,assumption:controller,assumption:controller-minimal} hold, and suppose the controller $c(z)$ is such that the algorithm in \Cref{fig:controller} asymptotically tracks a critical trajectory with rate $\rho\in(0,1)$ for all quadratic objectives~\eqref{eq:objective} satisfying these assumptions. Since the controller is minimal degree (by \Cref{assumption:controller-minimal}), it must have the form $c(z) = d(z)/m(z)$, where $d(z)$ and $m(z)$ are polynomials, the degree of the numerator $d(z)$ is strictly less than that of the denominator $m(z)$, and the model $m(z)$ is monic of degree $n$ with all roots on the unit circle. Define the characteristic polynomial of the subsystem in \eqref{eq:reduced-system} for a general eigenvalue $\lambda\in[\mu,L]$ as
\[
    p_\lambda(z) \defeq m(z) - \lambda\,d(z).
\]
Since the degree of $d(z)$ is strictly less than that of $m(z)$ and the model is monic, $p_\lambda(z)$ is also monic of degree $n$. Our goal is then to construct a lower bound on the convergence rate $\rho$ such that all roots of the polynomial $p_\lambda(z)$ are in the disk $\disk_{\leq\rho} \defeq \{z : |z|\leq\rho\}$ for all $\lambda\in[\mu,L]$.

Write the numerator of the optimization filter in terms of its coefficients as $d(z) = d_{n-1} z^{n-1} + \ldots + d_0$, where each of the $d_i$ coefficients may be zero. Also, denote the roots of the characteristic polynomial as $\zeta_1,\ldots,\zeta_n\in\complex$, which all have modulus at most $\rho$ by assumption. In terms of its roots and the coefficients of the controller polynomials, the characteristic polynomial is then
\[
    p_\lambda(z) \deq 
    z^n + \sum_{i=0}^{n-1} (m_i - \lambda\,d_i) z^i
    \deq \prod_{i=1}^n (z-\zeta_i).
\]
Define the corresponding elementary symmetric polynomials
\[
    e_k(p_\lambda) \deq \sum_{1\leq i_1<i_2<\cdots<i_k\leq n} \zeta_{i_1}\zeta_{i_2}\cdots\zeta_{i_k}
\]
for $k=1,\ldots,n$, which is the sum of all distinct products of $k$ roots. Vieta's formulas relate the characteristic polynomial coefficients to the elementary symmetric polynomials by
\[
    e_k(p_\lambda) \deq (-1)^k (m_{n-k}-\lambda d_{n-k}).
\]
Since $e_k(p_\lambda)$ is a sum of products of $k$ roots, each root has modulus at most~$\rho$, and there are $c_k \defeq \binom{n}{k}$ terms in the summation, the elementary symmetric polynomials satisfy the bound
\begin{equation}\label{eq:bound}
    |e_k(p_\lambda)| \leq c_k \rho^k, \qquad \forall \lambda\in[\mu,L] \text{ and } k=1,\ldots,n.
\end{equation}
We can then bound the size of the numerator polynomial coefficients by
\begin{align*}
    (L-\mu) |d_{n-k}| &= |(m_{n-k}-\mu\,d_{n-k})-(m_{n-k}-L\,d_{n-k})|, \\
    &\leq |m_{n-k}-\mu\,d_{n-k}| + |m_{n-k}-L\,d_{n-k}|, \\
    &\leq 2 c_k \rho^k,
\end{align*}
where the first inequality follows from the triangle inequality and the second inequality follows from the bound~\eqref{eq:bound} applied to both endpoints $\lambda=\mu$ and $\lambda=L$. Again using the bound~\eqref{eq:bound} at the endpoint $\lambda=\mu$ along with the (reverse) triangle inequality, we bound the size of the coefficients by
\[
    |m_{n-k}| - \mu\,|d_{n-k}| \leq |m_{n-k} - \mu\,d_{n-k}| \leq c_k \rho^k.
\]
Combining this with the bound on $|d_{n-k}|$ above yields
\begin{align*}
    |m_{n-k}| &\leq c_k\rho^k + \mu\,|d_{n-k}|, \\
    &\leq c_k\rho^k + \mu\frac{2c_k}{L-\mu}\rho^k, \\
    &= c_k \rho^k \frac{\kappa+1}{\kappa-1},
\end{align*}
where $\kappa\defeq L/\mu$ is the condition ratio. Isolating the convergence rate and using that this holds for all $k=1,\ldots,n$ yields the lower bound
\begin{equation}\label{eq:bound-general}
    \rho \dgeq \max_{k=1,\ldots,n} \ \biggl(\frac{|m_{n-k}|}{c_k} \cdot \frac{\kappa-1}{\kappa+1}\biggr)^{1/k}.
\end{equation}
Again using Vieta's formulas, each model coefficient $m_{n-k}$ is (up to a sign) the sum of all distinct products of $k$ roots, each of which has unit modulus by \Cref{assumption:model}. Therefore, $|m_{n-k}| = e_k(m) \leq c_k$ for all $k=1,\ldots,n$. Moreover, this holds with equality when $k=n$, which produces the lower bound in~\eqref{eq:TV}.
\end{proof}

Theorem~\ref{thm:lower_bound} provides a fundamental lower bound on 
the rate of convergence attainable by any minimal-order optimization 
filter. The bound is expressed solely in terms of the properties of 
the optimization (the condition ratio $\kappa$ of the objective) and 
the degree $n$ of the temporal variability model $m(z)$. 
The dependence on $\kappa$ is classical and in line with  
time-invariant counterparts: as the problem becomes more 
ill-conditioned (i.e., $\kappa$ grows), the lower bound approaches 
one, indicating arbitrarily slow convergence in the worst case. 
The role of $n$ quantifies a new, fundamental bound intrinsic to 
time-varying problems: as the temporal variability to be tracked 
becomes more complex (i.e., as \(n\) increases), the algorithm’s 
convergence rate necessarily degrades, highlighting that {\it 
high-order internal models inherently preclude fast convergence.}
Thus,~\eqref{eq:TV} characterizes an intrinsic tradeoff between 
problem conditioning and model complexity, highlighting that even in 
the best-case design, the convergence rate cannot be improved beyond 
this limit.
To compare, for the special case of polynomial models of the form $m(z) = (z-1)^n$, the lower bound for \textit{non-minimal} controllers is~\cite{AXW-IRP-IS:25}
\[
    \rho \geq \biggl(\frac{\sqrt{\kappa}-1}{\sqrt{\kappa}+1}\biggr)^{1/n}.
\]

\section{Controller Design via Root Locus}\label{sec:design}

We now shift to designing optimization filters of minimal order (see \Cref{def:minimal-order}) that satisfy \Cref{assumption:controller} and 
attain the fundamental rate bound~\eqref{eq:TV}. Our approach is based on 
root locus techniques, which enables us to construct closed-form solutions 
to the optimization~\eqref{eq:optimization-characterization_rho}. 
Due to the complexity of this task, we focus on three specific cases: $n=1, 2, 3.$

\subsection{Case: \texorpdfstring{$n=1$}{n=1}}

\begin{figure}[htb]
  \centering\includegraphics[width=\linewidth]{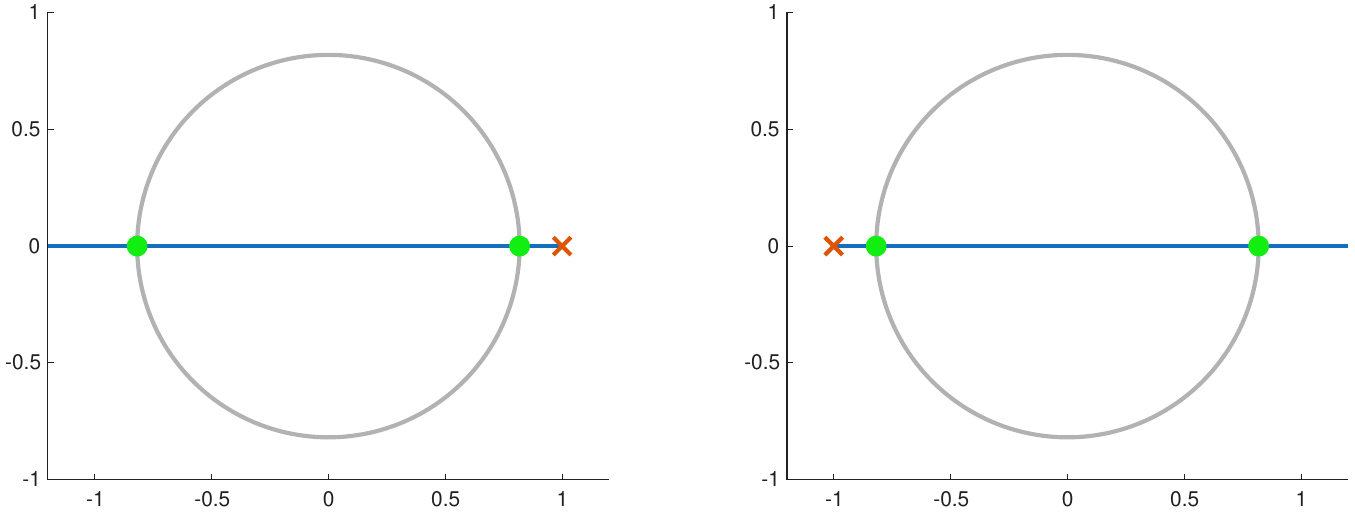}
  \caption{Root locus with controller $-\alpha/(z-1)$ (left) and $\alpha/(z+1)$ (right). The locus (\textcolor{matlabblue}{blue}) starts at the open-loop poles (\textcolor{matlaborange}{$\times$}). The pole locations at gains $\lambda=\mu$ and $\lambda=L$ are shown (\textcolor{mygreen}{$\bullet$}). For all $\lambda\in[\mu,L]$, the root locus is entirely contained in the $\rho$ circle (\textcolor{gray}{gray}).}
  \label{fig:rl-n1}
\end{figure}

When the model has only a single pole, it must be either $\pm 1$ since the poles are on the unit circle and the model coefficients are real (so complex roots must appear in conjugate pairs). We now consider these two cases.

First, suppose the model is $m(z) = z-1$. Then we can use the standard gradient descent controller
\[
  c(z) = \frac{-\alpha}{z-1} \ddand
  \alpha = \frac{2}{L+\mu},
\]
which achieves the optimal worst-case rate $\rho = \frac{\kappa-1}{\kappa+1}$. In this case, the root locus of $1-\lambda\,c(z)$ starts at the open loop pole of $z=1$ when $\lambda=0$ and moves to the left on the real axis as $\lambda$ increases, crossing $z=\rho$ when $\lambda=\mu$ and $z=-\rho$ when $\lambda=L$. 
If the model is $m(z) = z+1$, then the controller $c(z) = \tfrac{\alpha}{z+1}$ with the same stepsize $\alpha$ achieves the same rate. The root locus for both cases is shown in \Cref{fig:rl-n1}.

\subsection{Case: \texorpdfstring{$n=2$}{n=2}}

Now suppose the model has a single pair of complex conjugate poles on the unit circle with angle $\theta$ so that $n=2$. We can then parameterize the controller as
\[
  c(z) = \frac{c_1 z-c_2}{z^2 - 2\cos(\theta) z + 1}
\]
with $c_1,c_2\in\real$. The parameters that yield the optimal worst-case rate in~\eqref{eq:TV} are
\[
  c_1 = \frac{-2\cos\theta}{L} \ddand
  c_2 = \frac{-2}{L+\mu},
\]
which are the unique solutions to the condition that the root locus pass through both $z=-\rho$ and $z=\rho$ when $\lambda=L$. The root locus of this controller is shown in \Cref{fig:rl-n2}.

The previous case assumes the roots are complex conjugates. If instead the model has real roots at $z = +1$ and $z=-1$ so that $m(z) = z^2-1$, the optimal worst-case rate is achieved by the controller parameters $c_1 = 0$ and $c_2 = \tfrac{2}{L+\mu}$, which can be derived from the condition that the root locus pass through both $z=-\rho$ and $z=\rho$ when $\lambda=\mu$.

\begin{figure}[htb]
  \centering\includegraphics[height=4.5cm]{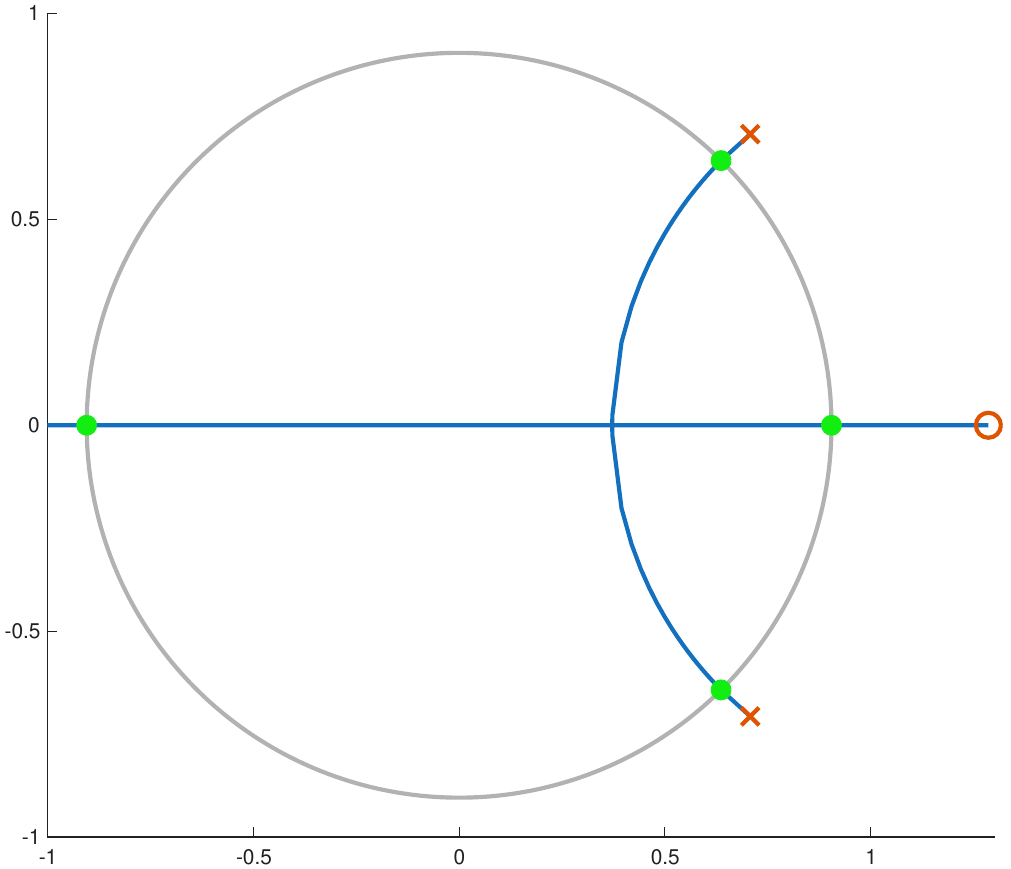}
  \caption{Root locus with controller for a single frequency $\theta = \pi/4$. The locus (\textcolor{matlabblue}{blue}) starts at the open-loop poles (\textcolor{matlaborange}{$\times$}) and ends at the open-loop zeros (\textcolor{matlaborange}{$\circ$}). The pole locations at gains $\lambda=\mu$ and $\lambda=L$ are shown (\textcolor{mygreen}{$\bullet$}). For all $\lambda\in[\mu,L]$, the root locus is entirely contained in the $\rho$ circle (\textcolor{gray}{gray}).}
  \label{fig:rl-n2}
\end{figure}

\subsection{Case: \texorpdfstring{$n=3$}{n=3}}\label{sec:design-n3}

Now suppose the model has both a pair of complex conjugate poles and a pole at one,
\begin{equation}\label{eq:model-n3}
  m(z) = (z-1) (z^2 - 2\cos(\theta) z + 1).
\end{equation}
We can then parameterize the controller as
\[
  c(z) = \frac{c_2 z^2 + c_1 z + c_0}{m(z)}.
\]
The optimal worst-case rate is achieved by the parameters
\[
  \begin{aligned}
    c_0 &= \tfrac{-1+\rho^3}{\mu}, \\
    c_1 &= \tfrac{(-L+\mu+(L+\mu)\rho)(1+\rho^2) + 2\rho(L+\mu-\rho(L-\mu))\cos(\theta)}{2\mu L\rho}, \\
    c_2 &= -\tfrac{(-L+\mu+(L+\mu)\rho^2)(1+\rho) + 2\rho(-L+\mu+(L+\mu)\rho)\cos(\theta)}{2\mu L\rho^2},
  \end{aligned}
\]
which yields the root locus in \Cref{fig:rl-n3}.

\begin{figure}[htb]
  \centering\includegraphics[height=4.5cm]{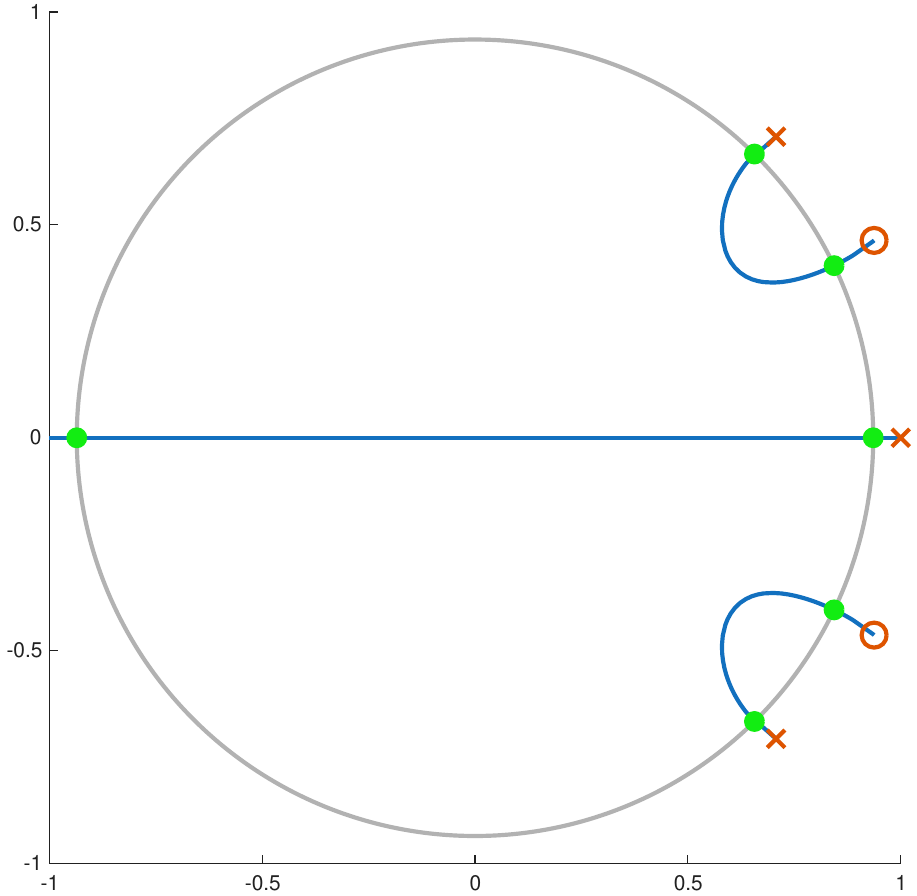}
  \caption{Root locus with controller for both a constant and a single frequency $\theta = \pi/4$. The locus (\textcolor{matlabblue}{blue}) starts at the open-loop poles (\textcolor{matlaborange}{$\times$}) and ends at the open-loop zeros (\textcolor{matlaborange}{$\circ$}). The pole locations at gains $\lambda=\mu$ and $\lambda=L$ are shown (\textcolor{mygreen}{$\bullet$}). For all $\lambda\in[\mu,L]$, the root locus is entirely contained in the $\rho$ circle (\textcolor{gray}{gray}).}
  \label{fig:rl-n3}
\end{figure}

\section{Numerical Simulation}\label{sec:simulation}

\begin{figure}[htb]
    \centering
    \includegraphics[width=\linewidth]{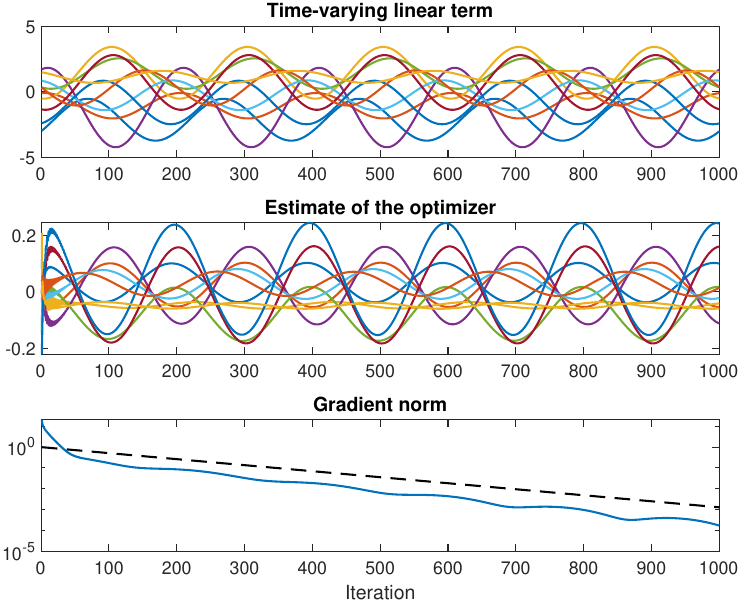}
    \caption{Numerical simulation of the optimal worst-case controller from \Cref{sec:design-n3} to minimize the quadratic objective in~\eqref{eq:objective}. (Top) Trajectories of each component in the time-varying linear term $b_k$. (Middle) Trajectories of each component in the time-varying optimizer $x_k^\star$. (Bottom) Gradient norm $\|A x_k + b_k\|$ and the bound on the worst-case rate $\rho^k$ (dashed). See \Cref{sec:simulation} for details.}
    \label{fig:simulation}
\end{figure}

To illustrate our results, we simulate the trajectory of~\eqref{eq:tf_00} using the optimization filter $C(z)$ designed in \Cref{sec:design} to the quadratic objective in~\eqref{eq:objective}. We chose parameters $\mu=1$, $L=10$, and $d=10$. The matrix $A$ was randomly generated with eigenvalues uniformly distributed in $[\mu,L]$, and the time-varying linear term $b_k$ was generated by an LTI system corresponding to the model $m(z)$ in~\eqref{eq:model-n3} with $\theta = 0.01\pi$. The system trajectories are shown in \Cref{fig:simulation}. Note that the controller from~\cite{NB-RC-SZ:24} does not exist in this case, as the linear matrix inequalities to design the controller were infeasible.


\section{Conclusions}

We have established fundamental limits on the attainable convergence 
rates of gradient-based algorithms for time-varying quadratic optimization. 
By leveraging tools from control theory, in particular the internal model 
principle and root locus techniques, we have shown that the optimal 
worst-case convergence rate necessarily degrades with the complexity of the 
temporal variability, quantified by the degree $n$ of the underlying model. 
Overall, the results of this paper contribute to a deeper understanding of 
the intrinsic performance limits of gradient-based methods in time-varying 
optimization, showing that temporal variability fundamentally constrains 
the achievable rate of convergence.
Avenues for future research are the incorporation of additional 
dynamics into the controller (beyond those strictly required by the internal 
model) to achieve \emph{acceleration}, and considering more general function classes 
beyond quadratics.




{\setlength\parskip{2pt}
\bibliographystyle{IEEEtran}
\bibliography{BIB/references,BIB/alias,BIB/full_GB,BIB/GB}}

\end{document}